\documentclass[psreqno,reqno,11pt]{amsart}
\usepackage{eurosym}
\usepackage{amssymb}
\usepackage{tikz}
\usepackage{amsmath}
\usepackage{tikz}
\usepackage{pgflibraryplotmarks}
\usepackage{wasysym}
\usepackage{stmaryrd}
\usepackage[english]{babel}

\usepackage{hyperref}

\theoremstyle{plain}
\newtheorem{theorem}{Theorem}[section]
\newtheorem{proposition}[theorem]{Proposition}
\newtheorem{lemma}[theorem]{Lemma}
\newtheorem{corollary}[theorem]{Corollary}

\theoremstyle{definition}
\newtheorem{definition}[theorem]{Definition}
\newtheorem{example}[theorem]{Example}

\begin{document}
\def\sect#1{\section*{\leftline{\large\bf #1}}}
\def\th#1{\noindent{\bf #1}\bgroup\it}
\def\endth{\egroup\par}

\title[The Classical Inequalities]{
Vector Lattices and $f$-Algebras: The Classical Inequalities}
\author{G. Buskes}
\address{Department of Mathematics, University of Mississippi, University,
MS 38677, USA}
\email{mmbuskes@olemiss.edu}
\author{C. Schwanke}
\address{Unit for BMI, North-West University, Private Bag X6001, Potchefstroom,
2520, South Africa}
\email{schwankc326@gmail.com}
\date{\today}

\begin{abstract}
We present some of the classical inequalities in analysis in the context of Archimedean (real or complex) vector lattices and $f$-algebras. In particular, we prove an identity for sesquilinear maps from the Cartesian square of a vector space to a geometric mean closed Archimedean vector lattice, from which a Cauchy-Schwarz inequality follows. A reformulation of this result for sesquilinear maps with a geometric mean closed semiprime Archimedean $f$-algebra as codomain is also given. In addition, a sufficient and necessary condition for equality is presented. We also prove a H\"older inequality for weighted geometric mean closed Archimedean $\Phi$-algebras, substantially improving results by Boulabiar and Toumi. As a consequence, a Minkowski inequality for weighted geometric mean closed Archimedean $\Phi$-algebras is obtained.
\end{abstract}

\maketitle

\section{Introduction}\label{S:intro}

Rich connections between the theory of Archimedean vector lattices and the classical inequalities in analysis, though hitherto little explored, were implied in \cite{BerHui} and the subsequent developments in \cite{Bo2,BusvR4,Kus2,Tou}. In particular, \cite{BusvR4} presents a relationship between the Cauchy-Schwarz inequality and the theory of multilinear maps on vector lattices, built on the analogy between disjointness in vector lattices and orthogonality in inner product spaces. The ideas in \cite{BusvR4} led to the construction of powers of vector lattices (see \cite{BoBus,BusvR2}), a theory that was recently extended to the complex vector lattice environment in \cite{BusSch2}. This paper follows the complex theme of \cite{BusSch2} and in fact contains results that are valid for both real vector lattices and complex vector lattices. We also conjoin the Cauchy-Schwarz, H\"older, and Minkowski inequalities with the theory of geometric mean closed Archimedean vector lattices, as found in \cite{Az,AzBoBus}.

We first discuss these results more closely.

In \cite[Corollary 4]{BusvR4}, the first author and van Rooij extend the classical Cauchy-Schwarz inequality as follows. If $V$ is a real vector space and $A$ is an Archimedean almost $f$-algebra then for every bilinear map $T\colon V\times V\rightarrow A$ such that
\begin{itemize}
	\item[(1)] $T(v,v)\geq 0\ (v\in V)$, and
	\item[(2)] $T(u,v)=T(v,u)\ (u,v\in V)$
\end{itemize}
we have
\[
T(u,v)^2\leq T(u,u)T(v,v)\ (u,v\in V).
\]
This is the classical Cauchy-Schwarz inequality when $A=\mathbb{R}$, which is equivalent to
\begin{equation}\label{(1)}
|T(u,v)|\leq\bigl(T(u,u)T(v,v)\bigr)^{1/2}=2^{-1}\inf\{\theta T(u,u)+\theta^{-1}T(v,v):\theta\in(0,\infty)\}
\end{equation}
for $u,v\in V$. 

The proof of \cite[Corollary 4]{BusvR4} easily adapts to a natural complex analogue for sesquilinear maps, though the condition for equality in the classical Cauchy-Schwarz inequality (see e.g. \cite[page 3]{Con}) does not hold in this more general context (see Example~\ref{E:noclaeqco}).

In Theorem~\ref{T:CSI} of this paper, we extend both the real and complex versions of the classical Cauchy-Schwarz inequality by replacing the codomain of the sesquilinear maps with an Archimedean (real or complex) vector lattice that is closed under the infimum in \eqref{(1)} above. We also prove a convenient formula for the difference between the two sides of the Cauchy-Schwarz inequality and use it to generalize the known condition for equality in the classical case. In Corollary~\ref{C:CSI}, we obtain a Cauchy-Schwarz inequality as well as a condition for equality for sesquilinear maps with values in a semiprime Archimedean $f$-algebra that is closed under the infimum in \eqref{(1)}.

Theorem~\ref{T:HI} of this paper proves a H\"older inequality for positive linear maps between Archimedean $\Phi$-algebras that are closed under certain weighted renditions of \eqref{(1)}. Our H\"older inequality generalizes \cite[Theorem 5, Corollary 6]{Bo2} and \cite[Theorem 3.12]{Tou} by (1) weakening the assumption of uniform completeness, (2) including irrational exponents via explicit formulas without restricting the codomain to the real numbers, (3) providing a result for several variables, and (4) enabling the domain and codomain of the positive linear maps in question to be either both real $\Phi$-algebras or both complex $\Phi$-algebras.

We remark that Theorem~\ref{T:HI} is itself a consequence of Proposition~\ref{P:Mali&Me}, which in turn generalizes a reformulation of the classical H\"older inequality by Maligranda \cite[(HI$_{1}$)]{Mali}. Indeed, Proposition~\ref{P:Mali&Me} is a reinterpretation of (HI$_{1}$) for Archimedean vector lattices that simultaneously extends (HI$_{1}$) to several variables. We add that Kusraev \cite[Theorem 4.2]{Kus2} independently developed his own version of (HI$_{1}$) in the setting of uniformly complete Archimedean vector lattices. Our Proposition~\ref{P:Mali&Me} contains \cite[Theorem 4.2]{Kus2}. Noting that Proposition~\ref{P:Mali&Me} relies primarily on the Archimedean vector lattice functional calculus, it (contrary to \cite[Theorem 4.2]{Kus2}) depends only on (at most) the countable Axiom of Choice.

Finally, we employ the H\"older inequality of Theorem~\ref{T:HI} to prove a Minkowski inequality in Theorem~\ref{T:MI}.

We proceed with some preliminaries.

\section{Preliminaries}\label{S:prelims}

We refer the reader to \cite{AB,LuxZan1,Zan2} for any unexplained terminology regarding vector lattices and $f$-algebras. Throughout, $\mathbb{R}$ is used for the real numbers, $\mathbb{C}$ denotes the complex numbers, $\mathbb{K}$ stands for either $\mathbb{R}$ or $\mathbb{C}$, and the symbol for the set of strictly positive integers is $\mathbb{N}$. 

An Archimedean real vector lattice $E$ is said to be \textit{square mean closed} (see \cite[page 482]{AzBoBus}) if $\sup\{ (\cos\theta)f+(\sin\theta)g:\theta\in[0,2\pi]\}$
exists in $E$ for every $f,g\in E$, and in this case we write
\[
f\boxplus g=\sup\{ (\cos\theta)f+(\sin\theta)g:\theta\in[0,2\pi]\}\ (f,g\in E).
\]
The notion of square mean closedness in vector lattices dates back to a 1973 paper by de Schipper, under the term \textit{property (E)} (see \cite[page 356]{dS}). We adopt de Schipper's definition of an Archimedean complex vector lattice but use the terminology found in \cite{AzBoBus}.

Throughout, $V+iV$ denotes the commonly used vector space complexification of a real vector space $V$. An \textit{Archimedean complex vector lattice} is a complex vector space of the form $E+iE$, where $E$ is a square mean closed Archimedean real vector lattice \cite[pages 356--357]{dS}.

An Archimedean real vector lattice will also be called an \textit{Archimedean vector lattice over $\mathbb{R}$}, and an Archimedean complex vector lattice will additionally be referred to as an \textit{Archimedean vector lattice over $\mathbb{C}$}. An \textit{Archimedean vector lattice over} $\mathbb{K}$ is a vector space that is either an Archimedean vector lattice over $\mathbb{R}$ or an Archimedean vector lattice over $\mathbb{C}$. Equivalently, an Archimedean vector lattice over $\mathbb{K}$ is a vector space over $\mathbb{K}$ that is equipped with an Archimedean modulus, as defined axiomatically by Mittelmeyer and Wolff in \cite[Definition 1.1]{MW}.

Given an Archimedean vector lattice $E+iE$ over $\mathbb{C}$, we write $\text{Re}(f+ig)=f$ and $\text{Im}(f+ig)=g\ (f,g\in E)$. For convenience, we write $\text{Re}(f)=f$ and $\text{Im}(f)=0\ (f\in E)$ when $E$ is an Archimedean vector lattice over $\mathbb{R}$. Lemma 1.2, Corollary 1.4, Proposition 1.5, and Theorem 2.2 of \cite{MW} together imply that the Archimedean modulus on an Archimedean vector lattice $E$ over $\mathbb{K}$ is given by the formula
\begin{align*}
|f|=\sup\{\text{Re}(\lambda f):\lambda\in\mathbb{K},|\lambda|=1\}\ (f\in E).
\end{align*}
In particular, for an Archimedean vector lattice $E$ over $\mathbb{R}$ we have $|f|=f\vee(-f)\ (f\in E)$, while $|f+ig|=f\boxplus g\ (f,g\in E)$ holds in any Archimedean vector lattice $E+iE$ over $\mathbb{C}$.

For an Archimedean vector lattice $E$ over $\mathbb{K}$, we define the \textit{positive cone} $E^{+}$ of $E$ by $E^{+}=\{ f\in E:|f|=f\}$, while the real vector lattice $E_{\rho}=\{ f-g:f,g\in E^{+}\}$ is called the \textit{real part} of $E$. With this notation, $E=E_{\rho}$ for every Archimedean vector lattice $E$ over $\mathbb{R}$, whereas $E=E_{\rho}+iE_{\rho}$ whenever $E$ is an Archimedean vector lattice over $\mathbb{C}$.

We say that an Archimedean vector lattice $E$ over $\mathbb{K}$ is \textit{uniformly complete} (respectively, \textit{Dedekind complete}) if $E_{\rho}$ is uniformly complete (respectively, Dedekind complete). We denote the Dedekind completion of an Archimedean vector lattice $E$ over $\mathbb{R}$ by $E^{\delta}$. Given an Archimedean vector lattice $E$ over $\mathbb{C}$, we define $E^{\delta}=(E_{\rho})^{\delta}+i(E_{\rho})^{\delta}$ and note that $E^{\delta}$ is also an Archimedean vector lattice over $\mathbb{C}$. Indeed, every Dedekind complete Archimedean vector lattice over $\mathbb{R}$ is uniformly complete \cite[Lemma 39.2, Theorem 39.4]{LuxZan1}, and every uniformly complete Archimedean vector lattice over $\mathbb{R}$ is square mean closed \cite[Section 2]{BeuHuidP}. 

An Archimedean vector lattice $E$ over $\mathbb{R}$ is said to be \textit{geometric mean closed} (see \cite[page 486]{AzBoBus}) if $\inf\{\theta f+\theta^{-1}g:\theta\in(0,\infty)\}$ exists in $E$ for every $f,g\in E^{+}$, and in this case we write
\[
f\boxtimes g=2^{-1}\inf\{\theta f+\theta^{-1}g:\theta\in(0,\infty)\}\ (f,g\in E^{+}).
\]
We define an Archimedean vector lattice over $\mathbb{K}$ to be \textit{geometric mean closed} (respectively, \textit{square mean closed}) if $A_{\rho}$ is geometric mean closed (respectively, square mean closed). Thus every Archimedean vector lattice over $\mathbb{C}$ is square mean closed. 

We next provide some basic information regarding Archimedean $f$-algebras that will be needed throughout this paper.

The multiplication on an Archimedean $f$-algebra $A$ canonically extends to a multiplication on $A+iA$. We call an Archimedean vector lattice $A$ over $\mathbb{K}$ an \textit{Archimedean $f$-algebra over $\mathbb{K}$} if $A_{\rho}$ is an $f$-algebra. If in addition $A$ has a multiplicative identity then we say that $A$ is an \textit{Archimedean $\Phi$-algebra over $\mathbb{K}$}. It was proved in \cite[Corollary 10.4]{dP} (also see \cite[Theorem 142.5]{Zan2}) that every Archimedean $\Phi$-algebra over $\mathbb{R}$ (and therefore every Archimedean $\Phi$-algebra over $\mathbb{K}$) is semiprime.

The multiplication on an Archimedean $f$-algebra over $\mathbb{K}$ will be denoted by juxtaposition throughout. For Archimedean $f$-algebras $A$ and $B$ over $\mathbb{K}$, we say that a map $T\colon A\rightarrow B$ is \textit{multiplicative} if $T(ab)=T(a)T(b)\ (a,b\in A)$.

Let $A$ be an Archimedean $f$-algebra over $\mathbb{K}$, and let $n\in\mathbb{N}$ and $a\in A^{+}$. If there exists a unique element $r$ of $A^{+}$ such that $r^{n}=a$, we write $r=a^{1/n}$ and say that $a^{1/n}$ exists. If $A$ is an Archimedean semiprime $f$-algebra and $a,r\in A^{+}$ satisfy $r^{n}=a$ then $r=a^{1/n}$ \cite[Proposition 2(ii)]{BeuHui}. Given $m,n\in\mathbb{N}$, we write $a^{m/n}=(a^{m})^{1/n}$, provided $(a^{m})^{1/n}$ exists.

Every uniformly complete semiprime Archimedean $f$-algebra $A$ over $\mathbb{R}$ is geometric mean closed (see \cite[Theorem 2.21]{Az}) and
\begin{equation}\label{(2)}
f\boxtimes g=(fg)^{1/2}\ (f,g\in A^{+}).
\end{equation}
The formula \eqref{(2)} also holds in the weaker condition when $A$ is geometric mean closed. In fact, the proof of \eqref{(2)} in \cite[Theorem 2.21]{Az} does not require uniform completeness, as illustrated in the next proposition. Since \cite{Az} is not widely accessible, we reproduce the proof of \cite[Theorem 2.21]{Az}, while (trivially) extending this theorem to include complex vector lattices.

\begin{proposition}\label{P:gmfa}
Let $A$ be a semiprime Archimedean $f$-algebra over $\mathbb{K}$. If $A$ is geometric mean closed then $f\boxtimes g=(fg)^{1/2}\ (f,g\in A^{+})$. 
\end{proposition}

\begin{proof}
Evidently, $A_{\rho}$ is a semiprime Archimedean $f$-algebra over $\mathbb{R}$. Let $f,g\in A^{+}$, and let $C$ be the $f$-subalgebra of $A_{\rho}$ generated by the elements $f,g,f\boxtimes g$. Suppose that $\phi\colon C\rightarrow\mathbb{R}$ is a nonzero multiplicative vector lattice homomorphism. Using \cite[Proposition 2.20]{Az} or \cite[Corollary 3.13]{BusSch} (first equality), we obtain
\begin{align*}
\phi(f\boxtimes g)=\phi(f)\boxtimes\phi(g)=\bigl(\phi(f)\phi(g)\bigr)^{1/2}=\bigl(\phi(fg)\bigr)^{1/2}.
\end{align*}
Therefore,
\begin{align*}
\phi\bigl((f\boxtimes g)^{2}\bigr)=\bigl(\phi(f\boxtimes g)\bigr)^{2}=\phi(fg).
\end{align*}
Since the set of all nonzero multiplicative vector lattice homomorphisms from $C$ into $\mathbb{R}$ separates the points of $C$ (see \cite[Corollary 2.7]{BusdPvR}), we have $(f\boxtimes g)^{2}=fg$.
\end{proof}

We conclude this section with some basic terminology regarding Archimedean vector lattices over $\mathbb{K}$.

Given an Archimedean vector lattice $E$ over $\mathbb{C}$, we define the complex conjugate
\begin{align*}
\overline{f+ig}=f-ig\ (f,g\in E_{\rho}).
\end{align*}
Since every Archimedean vector lattice over $\mathbb{R}$ canonically embeds into an Archimedean vector lattice over $\mathbb{C}$ (see \cite[Theorem 3.3]{BusSch2}), the previous definition also makes sense in Archimedean vector lattices over $\mathbb{R}$ (via such an embedding). If $E$ is an Archimedean vector lattice over $\mathbb{K}$ then the familiar identities $\text{Re}(f)=2^{-1}(f+\bar{f})$ and $\text{Im}(f)=(2i)^{-1}(f-\bar{f})$ are valid for every $f\in E$.

Let $V$ be a vector space over $\mathbb{K}$, and suppose that $F$ is an Archimedean vector lattice over $\mathbb{K}$. A map $T\colon V\times V\rightarrow F$ is called \textit{positive semidefinite} if $T(v,v)\geq 0$ for every $v\in V$. If $T(u,v)=\overline{T(v,u)}$ for each $u,v\in V$ then $T$ is said to be \textit{conjugate symmetric}. We say that $T$ is \textit{sesquilinear} if
\begin{itemize}
\item[(1)] $T(\alpha u_{1}+\beta u_{2},v)=\alpha T(u_{1},v)+\beta T(u_{2},v)\ (\alpha,\beta\in\mathbb{K}, u_{1},u_{2},v\in V)$, and
\item[(2)] $T(u,\alpha v_{1}+\beta v_{2})=\bar{\alpha}T(u,v_{1})+\bar{\beta}T(u,v_{2})\ (\alpha,\beta\in\mathbb{K}, u,v_{1},v_{2}\in V)$.
\end{itemize}

\section{A Cauchy-Schwarz Inequality}\label{S:CSI}

We prove a Cauchy-Schwarz inequality for sesquilinear maps with a geometric mean closed Archimedean vector lattice over $\mathbb{K}$ as codomain (Theorem~\ref{T:CSI}) and with a geometric mean closed semiprime Archimedean $f$-algebra over $\mathbb{K}$ as codomain (Corollary~\ref{C:CSI}). A necessary and sufficient condition for equality in Theorem~\ref{T:CSI} and Corollary~\ref{C:CSI} is given via an explicit formula for the difference between the two sides in the Cauchy-Schwarz inequality of Theorem 3.1 mentioned above. However, Example~\ref{E:noclaeqco} illustrates that the condition for equality in the classical Cauchy-Schwarz inequality fails in Theorem~\ref{T:CSI} and Corollary~\ref{C:CSI}.

We proceed to the main result of this section.

\begin{theorem}\label{T:CSI} \textnormal{\textbf{(Cauchy-Schwarz Inequality)}}
Let $V$ be a vector space over $\mathbb{K}$, and suppose that $F$ is a geometric mean closed Archimedean vector lattice over $\mathbb{K}$. If a map $T\colon V\times V\rightarrow F$ is positive semidefinite, conjugate symmetric, and sesquilinear then
\begin{itemize}
	\item[(1)] $\underset{z\in\mathbb{K}\setminus\{ 0\}}{\inf}\{|z|^{-1}T(zu-v,zu-v)\}$ exists in $F\ (u,v\in V)$,
	\item[(2)] $|T(u,v)|=T(u,u)\boxtimes T(v,v)-2^{-1}\underset{z\in\mathbb{K}\setminus\{ 0\}}{\inf}\{|z|^{-1}T(zu-v,zu-v)\}\ (u,v\in V)$,
	\item[(3)] $|T(u,v)|\leq T(u,u)\boxtimes T(v,v)\ (u,v\in V)$, and
	\item[(4)] $|T(u,v)|=T(u,u)\boxtimes T(v,v)$ if and only if $\underset{z\in\mathbb{K}\setminus\{ 0\}}{\inf}\{|z|^{-1}T(zu-v,zu-v)\}=0$.
\end{itemize}
\end{theorem}

\begin{proof}
Suppose that $T\colon V\times V\rightarrow F$ is a positive semidefinite, conjugate symmetric, sesquilinear map. Consider $T$ as a map from $V\times V$ to $F^{\delta}$. Let $u,v\in V$, and put $\theta\in (0,\infty)$. Using the sesquilinearity of $T$ and the identity $\text{Re}(f)=2^{-1}(f+\bar{f})$ for $f\in F^{\delta}$, we obtain
\begin{align*}
T(\theta u-v,u-\theta^{-1}v)&=T(\theta u,u)+T(v,\theta^{-1}v)-T(\theta u,\theta^{-1}v)-T(v,u)\\
&=\theta T(u,u)+\theta^{-1}T(v,v)-2\text{Re}\bigl(T(u,v)\bigr).
\end{align*}
Therefore,
\begin{align*}
\text{Re}\bigl(T(u,v)\bigr)=2^{-1}\bigl(\theta T(u,u)+\theta^{-1}T(v,v)\bigr)-(2\theta)^{-1}T(\theta u-v,\theta u-v).
\end{align*}
In particular, for every $\lambda\in S=\{\lambda\in\mathbb{K}:|\lambda|=1\}$ we have
\begin{align*}
\text{Re}\bigl(\lambda T(u,v)\bigr)=\text{Re}\bigl(T(\lambda u,v)\bigr)=2^{-1}\bigl(\theta T(u,u)+\theta^{-1}T(v,v)\bigr)-(2\theta)^{-1}T(\theta\lambda u-v,\theta\lambda u-v).
\end{align*}
Thus we obtain
\begin{align*}
|T(u,v)|&=\underset{\lambda\in S}{\sup}\bigl\{\text{Re}\bigl(\lambda T(u,v)\bigr)\bigr\}\\
&=2^{-1}\underset{\lambda\in S}{\sup}\bigl\{\theta T(u,u)+\theta^{-1}T(v,v)-\theta^{-1}T(\theta\lambda u-v,\theta\lambda u-v)\bigr\},
\end{align*}
where the suprema above are in $F^{\delta}$. Since $-\theta^{-1}T(\theta\lambda u-v,\theta\lambda u-v)\leq 0\ (\lambda\in S)$, it follows that $\underset{\lambda\in S}{\sup}\bigl\{-\theta^{-1}T(\theta\lambda u-v,\theta\lambda u-v)\bigr\}$ exists in $F^{\delta}$. Moreover,
\begin{align*}
\underset{\lambda\in S}{\sup}&\bigl\{\theta T(u,u)+\theta^{-1}T(v,v)-\theta^{-1}T(\theta\lambda u-v,\theta\lambda u-v)\bigr\}\\
&=\theta T(u,u)+\theta^{-1}T(v,v)+\underset{\lambda\in S}{\sup}\bigl\{-\theta^{-1}T(\theta\lambda u-v,\theta\lambda u-v)\bigr\}\\
&=\theta T(u,u)+\theta^{-1}T(v,v)-\underset{\lambda\in S}{\inf}\bigl\{ \theta^{-1}T(\theta\lambda u-v,\theta\lambda u-v)\bigr\}.
\end{align*}
Therefore,
\begin{align*}
2^{-1}\bigl(\theta T(u,u)+\theta^{-1}T(v,v)\bigr)=|T(u,v)|+2^{-1}\underset{\lambda\in S}{\inf}\bigl\{\theta^{-1}T(\theta\lambda u-v,\theta\lambda u-v)\bigr\}.
\end{align*}
Since $\underset{\lambda\in S}{\inf}\bigl\{\theta^{-1}T(\theta\lambda u-v,\theta\lambda u-v)\bigr\}\geq 0\ (\theta\in(0,\infty))$, it follows that
\begin{align*}
\underset{\theta\in(0,\infty)}{\inf}\Bigl\{\underset{\lambda\in S}{\inf}\{\theta^{-1}T(\theta\lambda u-v,\theta\lambda u-v)\}\Bigr\}
\end{align*}
exists in $F^{\delta}$. Then with the infima still in $F^{\delta}$,

\begin{align*}
T(u,u)\boxtimes T(v,v)&=\underset{\theta\in(0,\infty)}{\inf}\Bigl\{|T(u,v)|+2^{-1}\underset{\lambda\in S}{\inf}\{\theta^{-1}T(\theta\lambda u-v,\theta\lambda u-v)\}\Bigr\}\\
&=|T(u,v)|+2^{-1}\underset{\theta\in(0,\infty)}{\inf}\Bigl\{\underset{\lambda\in S}{\inf}\{\theta^{-1}T(\theta\lambda u-v,\theta\lambda u-v)\}\Bigr\}\\
&=|T(u,v)|+2^{-1}\underset{\lambda\in S,\theta\in(0,\infty)}{\inf}\{\theta^{-1}T(\theta\lambda u-v,\theta\lambda u-v)\}\\
&=|T(u,v)|+2^{-1}\underset{z\in\mathbb{K}\setminus\{ 0\}}{\inf}\{|z|^{-1}T(zu-v,zu-v)\}.
\end{align*}
Thus, now in $F$, the equality
\begin{align*}
2^{-1}\underset{z\in\mathbb{K}\setminus\{ 0\}}{\inf}\{|z|^{-1}T(zu-v,zu-v)\}=\bigl(T(u,u)\boxtimes T(v,v)-|T(u,v)|\bigr)\
\end{align*}
authenticates statement (1) of this theorem and implies (in $F$)
\begin{align*}
|T(u,v)|=T(u,u)\boxtimes T(v,v)-2^{-1}\underset{z\in\mathbb{K}\setminus\{ 0\}}{\inf}\{|z|^{-1}T(zu-v,zu-v)\}.
\end{align*}
Therefore, part (2) is verified. Statements (3) and (4) immediately follow from (2). 
\end{proof}

As a consequence of Theorem~\ref{T:CSI} and Proposition~\ref{P:gmfa}, we obtain the following.

\begin{corollary}\label{C:CSI}
Let $V$ be a vector space over $\mathbb{K}$, and suppose that $A$ is a geometric mean closed semiprime Archimedean $f$-algebra over $\mathbb{K}$. If $T\colon V\times V\rightarrow A$ is a positive semidefinite, conjugate symmetric, sesquilinear map then
\begin{itemize}
	\item[(1)] $\underset{z\in\mathbb{K}\setminus\{ 0\}}{\inf}\{|z|^{-1}T(zu-v,zu-v)\}$ exists in $A\ (u,v\in V)$,
	\item[(2)] $|T(u,v)|^{2}=\Bigl(\bigl(T(u,u)T(v,v)\bigr)^{1/2}-2^{-1}\underset{z\in\mathbb{K}\setminus\{ 0\}}{\inf}\{|z|^{-1}T(zu-v,zu-v)\}\Bigr)^{2}\ (u,v\in V)$,
	\item[(3)] $|T(u,v)|^{2}\leq T(u,u)T(v,v)\ (u,v\in V)$, and
	\item[(4)] $|T(u,v)|^{2}=T(u,u)T(v,v)$ if and only if $\underset{z\in\mathbb{K}\setminus\{ 0\}}{\inf}\{|z|^{-1}T(zu-v,zu-v)\}=0$. 
\end{itemize}
\end{corollary}

Statement (3) in Corollary~\ref{C:CSI} follows from part (2) of the corollary as well as \cite[Proposition 2(iii)]{BeuHui}. For $\mathbb{K}=\mathbb{R}$, this statement is contained in \cite[Corollary 4]{BusvR4}. Similarly, the given statement is contained in the complex analogue of \cite[Corollary 4]{BusvR4} (mentioned in the introduction) when $\mathbb{K}=\mathbb{C}$. Part (2) of Corollary~\ref{C:CSI}, however, depends on the uniqueness of square roots, which implies the semiprime property in Archimedean almost $f$-algebras. Indeed, let $A$ be an Archimedean almost $f$-algebra and suppose $a^{2}=b^{2}$ implies $a=b\ (a,b\in A^{+})$. Let $a$ be a nilpotent in $A$. Then $a^{3}=0$ (see \cite[Theorem 3.2]{BerHui2}), and thus $a^{4}=0$. Using the uniqueness of square roots twice, we obtain $a=0$. Finally, every semiprime almost $f$-algebra is automatically an $f$-algebra (\cite[Theorem 1.11(i)]{BerHui2}).

The special case where $A=\mathbb{K}$ in the inequality of Corollary~\ref{C:CSI} is the classical Cauchy-Schwarz inequality. Thus we know in this special case that $|T(u,v)|^{2}=T(u,u)T(v,v)$ if and only if there exist $\alpha,\beta\in\mathbb{K}$, not both zero, such that $T(\beta u+\alpha v,\beta u+\alpha v)=0$ (see, e.g., \cite[page 3]{Con}). This criterion no longer holds for Theorem~\ref{T:CSI} nor Corollary~\ref{C:CSI}.

\begin{example}\label{E:noclaeqco}
Define $T\colon \mathbb{K}^{2}\times\mathbb{K}^{2}\rightarrow\mathbb{K}^{2}$ by
\begin{align*}
T\bigl((z_{1},z_{2}),(w_{1},w_{2})\bigr)=(z_{1}\bar{w_{1}},z_{2}\bar{w_{2}})\ \bigl((z_{1},z_{2}),(w_{1},w_{2})\in\mathbb{K}^{2}\bigr).
\end{align*}
Since $\mathbb{C}^{2}=\mathbb{R}^{2}+i\mathbb{R}^{2}$, we see that $\mathbb{K}^{2}$ is a geometric mean closed semiprime Archimedean $f$-algebra over $\mathbb{K}$ with respect to the coordinatewise vector space operations, coordinatewise ordering, and coordinatewise multiplication. Also, $T$ is a positive semidefinite, conjugate symmetric, sesquilinear map. Note that 
\[
|T\bigl((1,0),(0,1)\bigr)|^{2}=T\bigl((1,0),(1,0)\bigr)T\bigl((0,1),(0,1)\bigr).
\]
Suppose there exist $\alpha,\beta\in\mathbb{K}$, not both zero, for which
\[
T\bigl(\beta(1,0)+\alpha(0,1),\beta(1,0)+\alpha(0,1)\bigr)=0.
\]
Then $(|\beta|^{2},|\alpha|^{2})=(0,0)$, a contradiction.
\end{example}

\section{A H\"older Inequality}\label{S:HI}

We prove a H\"older inequality for positive linear maps between weighted geometric mean closed Archimedean $\Phi$-algebras over $\mathbb{K}$ in this section, extending \cite[Theorem 5, Corollary 6]{Bo2} by Boulabiar and \cite[Theorem 3.12]{Tou} by Toumi. We begin with some definitions.

Let $A$ be an Archimedean $f$-algebra over $\mathbb{K}$, and suppose that $n\in\mathbb{N}$. As usual, we write $\prod\limits_{k=1}^{n}a_{k}=a_{1}\cdots a_{n}$ for $a_{1},\dots,a_{n}\in A$.

For every $r_{1},\dots,r_{n}\in(0,1)$ such that $\sum\limits_{k=1}^{n}r_{k}=1$, we define a \textit{weighted geometric mean} $\gamma_{r_{1},\dots,r_{n}}:\mathbb{R}^{n}\rightarrow\mathbb{R}$ by
\begin{align*}
\gamma_{r_{1},\dots,r_{n}}(x_{1},\dots,x_{n})=\prod\limits_{k=1}^{n}|x_{k}|^{r_{k}}\ (x_{1},\dots,x_{n}\in\mathbb{R}).
\end{align*}
The weighted geometric means are concave on $(\mathbb{R}^{+})^{n}$ as well as continuous and positively homogeneous on $\mathbb{R}^{n}$. Moreover, for each $r_{1},\dots,r_{n}\in(0,1)$ with $\sum\limits_{k=1}^{n}r_{k}=1$, it follows from \cite[Lemma 3.6(iii)]{BusSch} that
\begin{align*}
\gamma_{r_{1},\dots,r_{n}}(x_{1},\dots,x_{n})=\inf\Bigl\{\sum\limits_{k=1}^{n}r_{k}\theta_{k}x_{k}:\theta_{k}\in(0,\infty),\ \prod\limits_{k=1}^{n}\theta_{k}^{r_{k}}=1\Bigr\}\ (x_{1},\dots,x_{n}\in\mathbb{R}^{+}).
\end{align*}

An Archimedean vector lattice $E$ over $\mathbb{K}$ is said to be \textit{weighted geometric mean closed} if $\inf\Bigl\{\sum\limits_{k=1}^{n}r_{k}\theta_{k}|f_{k}|:\theta_{k}\in(0,\infty),\ \prod\limits_{k=1}^{n}\theta_{k}^{r_{k}}=1\Bigr\}$ exists in $E$ for every $f_{1},\dots,f_{n}\in E$ and every $r_{1},\dots,r_{n}\in(0,1)$ with $\sum\limits_{k=1}^{n}r_{k}=1$. In this case, we write
\begin{align*}
\bigtriangleup\limits_{k=1}^{n}(f_{k},r_{k})=\inf\Bigl\{\sum\limits_{k=1}^{n}r_{k}\theta_{k}|f_{k}|:\theta_{k}\in(0,\infty),\ \prod\limits_{k=1}^{n}\theta_{k}^{r_{k}}=1\Bigr\}\ (f_{1},\dots,f_{n}\in E).
\end{align*}

Let $(X,\mathcal{M},\mu)$ be a measure space, and let $p\in(1,\infty)$. It follows from \cite[(HI$_{1}$)]{Mali} by Maligranda that $|f|^{1/p}|g|^{1-1/p}\in L_{1}(X,\mu)$ for $f,g\in L_{1}(X,\mu)$, and
\begin{align*}
||\ |f|^{1/p}|g|^{1-1/p}\ ||_{1}\leq||f||_{1}^{1/p}||g||_{1}^{1-1/p}.
\end{align*}
Following Maligranda's proof, we redevelop and extend \cite[(HI$_{1}$)]{Mali} to a multivariate version in the setting of positive operators between vector lattices.

\begin{proposition}\label{P:Mali&Me}
Let $E$ and $F$ be weighted geometric mean closed Archimedean vector lattices over $\mathbb{K}$, and suppose that $r_{1},\dots,r_{n}\in(0,1)$ satisfy $\sum\limits_{k=1}^{n}r_{k}=1$.
\begin{itemize}
\item[(1)] For each positive linear map $T\colon E\rightarrow F$,
\begin{align*}
T\Bigl(\bigtriangleup\limits_{k=1}^{n}(f_{k},r_{k})\Bigr)\leq\bigtriangleup\limits_{k=1}^{n}\bigl(T(|f_{k}|),r_{k}\bigr)\ (f_{1},\dots,f_{n}\in E).
\end{align*}
\item[(2)] If $T\colon E\rightarrow F$ is a linear map then $T$ is a vector lattice homomorphism if and only if
\begin{align*}
T\Bigl(\bigtriangleup\limits_{k=1}^{n}(f_{k},r_{k})\Bigr)=\bigtriangleup\limits_{k=1}^{n}\bigl(T(f_{k}),r_{k}\bigr)\ (f_{1},\dots,f_{n}\in E).
\end{align*}
\item[(3)] If $G$ is a (not necessarily weighted geometric mean closed) vector sublattice of $E$, $T\colon G\rightarrow F$ is a vector lattice homomorphism, and $f_{1},\dots,f_{n},\bigtriangleup\limits_{k=1}^{n}(f_{k},r_{k})\in G$ then
\[
T\Bigl(\bigtriangleup\limits_{k=1}^{n}(f_{k},r_{k})\Bigr)=\bigtriangleup\limits_{k=1}^{n}\bigl(T(f_{k}),r_{k}\bigr).
\]
\end{itemize}
\end{proposition}

\begin{proof}
(1) Assume $T\colon E\rightarrow F$ is a positive linear map. Let $f_{1},\dots,f_{n}\in E$, and suppose $\theta_{1},\dots,\theta_{n}\in(0,\infty)$ are such that $\prod\limits_{k=1}^{n}\theta_{k}^{r_{k}}=1$. From the positivity and linearity of $T$ we have
\begin{align*}
T\Bigl(\bigtriangleup\limits_{k=1}^{n}(f_{k},r_{k})\Bigr)\leq T\Bigl(\sum\limits_{k=1}^{n}r_{k}\theta_{k}|f_{k}|\Bigr)=\sum\limits_{k=1}^{n}r_{k}\theta_{k}T(|f_{k}|).
\end{align*}
Then
\begin{align*}
T\Bigl(\bigtriangleup\limits_{k=1}^{n}(f_{k},r_{k})\Bigr)\leq\inf\Bigl\{\sum\limits_{k=1}^{n}r_{k}\theta_{k}T(|f_{k}|):\theta_{k}\in(0,\infty),\ \prod\limits_{k=1}^{n}\theta_{k}^{r_{k}}=1\Bigr\}.
\end{align*}

(2) Suppose $T\colon E\rightarrow F$ is a vector lattice homomorphism. It follows from  \cite[Theorem 3.7(2)]{BusSch} that $\gamma_{r_{1},\dots,r_{n}}(f_{1},\dots,f_{n})$, which is defined via functional calculus \cite[Definition 3.1]{BusdPvR}, exists in $E$ for every $f_{1},\dots,f_{n}\in E^{+}$ and
\[
\gamma_{r_{1},\dots,r_{n}}(f_{1},\dots,f_{n})=\bigtriangleup\limits_{k=1}^{n}(f_{k},r_{k})\ (f_{1},\dots,f_{n}\in E^{+}).
\]
It is readily checked using \cite[Definition 3.1]{BusdPvR} and the identity
\[
\gamma_{r_{1},\dots,r_{n}}(x_{1},\dots,x_{n})=\gamma_{r_{1},\dots,r_{n}}(|x_{1}|,\dots,|x_{n}|)\ (x_{1},\dots,x_{n}\in\mathbb{R})
\]
that $\gamma_{r_{1},\dots,r_{n}}(f_{1},\dots,f_{n})$ exists in $E$ for all $f_{1},\dots,f_{n}\in E$ and
\[
\gamma_{r_{1},\dots,r_{n}}(f_{1},\dots,f_{n})=\gamma_{r_{1},\dots,r_{n}}(|f_{1}|,\dots,|f_{n}|)\ (f_{1},\dots,f_{n}\in E).
\]
It follows that
\[
\gamma_{r_{1},\dots,r_{n}}(f_{1},\dots,f_{n})=\bigtriangleup\limits_{k=1}^{n}(f_{k},r_{k})\ (f_{1},\dots,f_{n}\in E).
\]
By \cite[Theorem 3.11]{BusSch} (second equality), we have for all $f_{1},\dots,f_{n}\in E$,
\begin{align*}
T\bigl(\bigtriangleup\limits_{k=1}^{n}(f_{k},r_{k})\bigr)&=T\bigl(\gamma_{r_{1},\dots,r_{n}}(f_{1},\dots,f_{n})\bigr)\\
&=\gamma_{r_{1},\dots,r_{n}}\bigl(T(f_{1}),\dots,T(f_{n})\bigr)=\bigtriangleup\limits_{k=1}^{n}\bigl(T(f_{k}),r_{k}\bigr).
\end{align*}
On the other hand, assume $T\colon E\rightarrow F$ is a linear map and
\begin{align*}
T\Bigl(\bigtriangleup\limits_{k=1}^{n}(f_{k},r_{k})\Bigr)=\bigtriangleup\limits_{k=1}^{n}\bigl(T(f_{k}),r_{k}\bigr)\ (f_{1},\dots,f_{n}\in E).
\end{align*}
From \cite[Theorem 3.11]{BusSch}, we conclude that
\begin{align*}
T\bigl(\gamma_{r_{1},\dots,r_{n}}(f_{1},\dots,f_{n})\bigr)=\gamma_{r_{1},\dots,r_{n}}\bigl(T(f_{1}),\dots,T(f_{n})\bigr)\ (f_{1},\dots,f_{n}\in E),
\end{align*}
and (since $\gamma_{r_{1},\dots,r_{n}}(x,\dots,x)=|x|\ (x\in\mathbb{R})$) that $T$ is a vector lattice homomorphism.

(3) Let $G$ be a vector sublattice of $E$, and let $T:G\rightarrow F$ be a vector lattice homomorphism. Let $\mathcal{D}$ be the collection of all weighted geometric means. That is, let
\[
\mathcal{D}=\left\{\gamma_{r_{1},\dots,r_{n}}:r_{1},\dots,r_{n}\in(0,1)\ \text{and}\ \sum\limits_{k=1}^{n}r_{k}=1\right\}.
\]
It follows from \cite[Theorem 3.7(2)]{BusSch} that $F$ is $\mathcal{D}$-complete (see \cite[Definition 3.2]{BusSch}). By \cite[Theorem 3.17]{BusSch}, $T$ uniquely extends to a vector lattice homomorphism $T^{\mathcal{D}}:G^{\mathcal{D}}\rightarrow F$, where $G^{\mathcal{D}}$ denotes the $\mathcal{D}$-completion of $G$ (see \cite[Definition 3.10]{BusSch}). Note that $G^{\mathcal{D}}$ is $\mathcal{D}$-complete by definition. Thus \cite[Theorem 3.7(2)]{BusSch} implies that $G^{\mathcal{D}}$ is weighted geometric mean closed. An appeal to (2) now verifies (3).
\end{proof}

Let $A$ be an Archimedean $\Phi$-algebra over $\mathbb{K}$. If $A$ is uniformly complete then $a^{1/n}$ exists in $A$ for every $a\in A^{+}$ and every $n\in\mathbb{N}$ (see \cite[Corollary 6]{BeuHui}). It follows that $a^{q}$ exists in $A$ for every $a\in A^{+}$ and every $q\in\mathbb{Q}\cap(0,\infty)$. The assumption of uniform completeness in \cite[Corollary 6]{BeuHui} can be weakened, which is the content of our next lemma.

\begin{lemma}\label{L:a^q}
Let $A$ be a weighted geometric mean closed Archimedean $\Phi$-algebra over $\mathbb{K}$. Then $a^{q}$ exists in $A$ for all $a\in A^{+}$ and $q\in\mathbb{Q}\cap(0,\infty)$. Furthermore, if $q_{1},\dots,q_{n}\in\mathbb{Q}\cap(0,1)$ are such that $\sum\limits_{k=1}^{n}q_{k}=1$ then for every $a_{1},\dots,a_{n}\in A^{+}$,
\begin{align*}
\prod\limits_{k=1}^{n}a_{k}^{q_{k}}=\bigtriangleup\limits_{k=1}^{n}(a_{k},q_{k})\in A.
\end{align*}
\end{lemma}

\begin{proof}
Denote the unit element of $A$ by $e$, and let $a\in A^{+}$. In order to prove that $a^{q}$ is defined in $A$ for every $q\in\mathbb{Q}\cap(0,\infty)$, it suffices to verify that $a^{1/n}$ exists in $A$ for all $n\in\mathbb{N}$. To this end, let $n\in\mathbb{N}\setminus\{1\}$. Let $C$ be the Archimedean $\Phi$-subalgebra of $A_{\rho}$ generated by  the elements $a\in A^{+}$ and
\[
b=\inf\{n^{-1}\theta_{1}a+(1-n^{-1})\theta_{2} e:\theta_{1},\theta_{2}\in(0,\infty),\ \theta_{1}^{1/n}\theta_{2}^{1-1/n}=1\}\in A^{+}.
\]
Suppose that $\omega\colon C\rightarrow\mathbb{R}$ is a nonzero multiplicative vector lattice homomorphism. It follows that $\omega(e)=1$. Using Proposition~\ref{P:Mali&Me}(3) (third equality), we obtain
\begin{align*}
\omega(b^{n})&=\omega(b)^{n}\\
&=\big(\omega(\inf\{n^{-1}\theta_{1}a+(1-n^{-1})\theta_{2} e:\theta_{1},\theta_{2}\in(0,\infty),\ \theta_{1}^{1/n}\theta_{2}^{1-1/n}=1\})\bigr)^{n}\\
&=\big(\inf\{n^{-1}\theta_{1}\omega(a)+(1-n^{-1})\theta_{2} :\theta_{1},\theta_{2}\in(0,\infty),\ \theta_{1}^{1/n}\theta_{2}^{1-1/n}=1\})\bigr)^{n}\\
&=\Bigl(\gamma_{\frac{1}{n},\frac{n-1}{n}}\bigl(\omega(a),1\bigr)\Bigr)^{n}=\Bigl(\bigl(\omega(a)\bigr)^{1/n}\Bigr)^{n}=\omega(a).
\end{align*}
Since the set of all nonzero multiplicative vector lattice homomorphisms $\omega\colon C\rightarrow\mathbb{R}$ separates the points of $C$ (see \cite[Corollary 2.7]{BusdPvR}), we have in $A_{\rho}$ that $b^{n}=a$. But then $a^{1/n}=b$.

Finally, a similar proof verifies that
\begin{align*}
\prod\limits_{k=1}^{n}a_{k}^{q_{k}}=\bigtriangleup\limits_{k=1}^{n}(a_{k},q_{k})
\end{align*}
for every $a_{1},\dots,a_{n}\in A^{+}$ and all $q_{1},\dots,q_{n}\in\mathbb{Q}\cap(0,1)$ such that $\sum_{k=1}^{n}q_{k}=1$.
\end{proof}

We next use the proof of Lemma~\ref{L:a^q} as a guide to define strictly positive irrational powers of positive elements in weighted geometric mean closed Archimedean $\Phi$-algebras in an intrinsic manner that does not require representation theory dependent on more than the countable Axiom of Choice. For $r\in(0,\infty)$, define
\begin{align*}
\lfloor r\rfloor=\max\{n\in\mathbb{N}\cup\{0\}:n\leq r\}\ \text{and}\ \tilde{r}=r-\lfloor r\rfloor.
\end{align*}

\begin{definition}\label{D:a^r}
Suppose that $A$ is a weighted geometric mean closed Archimedean $\Phi$-algebra over $\mathbb{K}$. Let $e$ be the unit element of $A$. For $a\in A^{+}$ and $r\in(0,\infty)$, define
\begin{align*}
a^{r}=a^{\lfloor r\rfloor}\inf\{\tilde{r}\theta_{1}a+(1-\tilde{r})\theta_{2} e:\theta_{1},\theta_{2}\in(0,\infty),\ \theta_{1}^{\tilde{r}}\theta_{2}^{1-\tilde{r}}=1\},
\end{align*}
where $a^{\lfloor r\rfloor}$ is taken to equal $e$ in the case where $\lfloor r\rfloor=0$.
\end{definition}

By Lemma~\ref{L:a^q}, the above definition of strictly positive real exponents extends the natural definition of strictly positive rational exponents previously discussed. We next give an easy corollary of Proposition~\ref{P:Mali&Me}(3).

\begin{corollary}\label{C:T(a^r)=T(a)^r}
Let $A$ and $B$ be weighted geometric mean closed Archimedean $\Phi$-algebras over $\mathbb{K}$ with unit elements $e$ and $e'$, respectively. Suppose $C$ is a $\Phi$-subalgebra of $A$ and that $T\colon C\rightarrow B$ is a multiplicative vector lattice homomorphism such that $T(e)=e'$. Let $a\in A^{+}$ and $r\in(0,\infty)$. If $a,a^{r}\in C$ then $T(a^{r})=\bigl(T(a)\bigr)^{r}$.
\end{corollary}

The following lemma verifies some familiar (and needed for Theorems~\ref{T:HI} and \ref{T:MI}) arithmetical rules for positive real exponents in geometric mean closed Archimedean $\Phi$-algebras over $\mathbb{K}$.

\begin{lemma}\label{L:powerrules}
Let $p,q\in(0,\infty)$, and let $A$ be a weighted geometric mean closed Archimedean $\Phi$-algebra over $\mathbb{K}$. For each $a\in A^{+}$, the following hold.
\begin{itemize}
	\item[(1)] $(a^{p})^{q}=a^{pq}$.
	\item[(2)] $a^{p}a^{q}=a^{p+q}$.
\end{itemize}
\end{lemma}

\begin{proof}
We prove (1), leaving the similar proof of (2) to the reader. To this end, let $a\in A^{+}$. Let $C$ be the real $\Phi$-subalgebra of $A_{\rho}$ generated by $a, a^{\tilde{p}}, (a^{p})^{\tilde{q}}$, and $a^{\overset{\sim}{pq}}$, and note that $a, a^{p}, (a^{p})^{q}, a^{pq}\in C$. If $\omega\colon C\rightarrow\mathbb{R}$ is a nonzero multiplicative vector lattice homomorphism then $\omega$ is surjective and thus $\omega(e)=1$. Using Corollary~\ref{C:T(a^r)=T(a)^r}, we obtain
\begin{align*}
\omega\bigl((a^{p})^{q}\bigr)&=\bigl(\omega(a^{p})\bigr)^{q}=\bigl(\omega(a)\bigr)^{pq}=\omega(a^{pq}).
\end{align*}
Since the nonzero multiplicative vector lattice homomorphisms separate the points of $C$ (see \cite[Corollary 2.7]{BusdPvR}), we conclude that $(a^{p})^{q}=a^{pq}$.
\end{proof}

In light of Definition~\ref{D:a^r}, the second part of Lemma~\ref{L:a^q} can now be improved to include irrational exponents. The proof of Lemma~\ref{L:a^r} uses real-valued multiplicative vector lattice homomorphisms, similar to what is found in the proofs of Lemmas \ref{L:a^q} and \ref{L:powerrules}. Therefore, the proof is omitted.

\begin{lemma}\label{L:a^r}
Let $A$ be a weighted geometric mean closed Archimedean $\Phi$-algebra over $\mathbb{K}$. If $r_{1},\dots,r_{n}\in(0,1)$ are such that $\sum\limits_{k=1}^{n}r_{k}=1$ then for every $a_{1},\dots,a_{n}\in A^{+}$,
\begin{align*}
\prod\limits_{k=1}^{n}a_{k}^{r_{k}}=\bigtriangleup\limits_{k=1}^{n}(a_{k},r_{k}).
\end{align*}
\end{lemma}

We proceed with the main theorem of this section.

\begin{theorem}\label{T:HI} \textnormal{\textbf{(H\"older Inequality)}}
Let $p_{1},\dots,p_{n}\in(1,\infty)$ with $\sum\limits_{k=1}^{n}p_{k}^{-1}=1$. Assume $A$ is a weighted geometric mean closed Archimedean $\Phi$-algebra over $\mathbb{K}$.
\begin{itemize}
\item[(1)] If $B$ is also a weighted geometric mean closed Archimedean $\Phi$-algebra over $\mathbb{K}$ and $T\colon A\rightarrow B$ is a positive linear map then
\begin{align*}
T\Bigl(\prod\limits_{k=1}^{n}|a_{k}|\Bigr)\leq\prod\limits_{k=1}^{n}\bigl(T(|a_{k}|^{p_{k}})\bigr)^{1/p_{k}}\ (a_{1},\dots,a_{n}\in A).
\end{align*}
\item[(2)] If $B$ is a weighted geometric mean closed Archimedean vector lattice over $\mathbb{K}$ and $T\colon A\rightarrow B$ is a positive linear map then
\begin{align*}
T\Bigl(\prod\limits_{k=1}^{n}|a_{k}|\Bigr)\leq\bigtriangleup\limits_{k=1}^{n}\bigl(T(|a_{k}|^{p_{k}}),1/p_{k}\bigr)\ (a_{1},\dots,a_{n}\in A).
\end{align*}
\end{itemize}
\end{theorem}

\begin{proof}
We only prove (1) since the proof of (2) is similar. To this end, let $B$ be a weighted geometric mean closed Archimedean $\Phi$-algebra over $\mathbb{K}$, and suppose $T\colon A\rightarrow B$ is a positive linear map. Using Lemma~\ref{L:powerrules}(1) (first equality), Lemma~\ref{L:a^r} (second equality and last equality), and Proposition~\ref{P:Mali&Me}(1) (for the inequality), we have
\begin{align*}
T\Bigl(\prod\limits_{k=1}^{n}|a_{k}|\Bigr)&=T\Bigl(\prod\limits_{k=1}^{n}(|a_{k}|^{p_{k}})^{1/p_{k}}\Bigr)=T\Bigl(\bigtriangleup\limits_{k=1}^{n}(|a_{k}|^{p_{k}},1/p_{k})\Bigr)\\
&\leq\bigtriangleup\limits_{k=1}^{n}\bigl(T(|a_{k}|^{p_{k}}),1/p_{k}\bigr)=\prod\limits_{k=1}^{n}\bigl(T(|a_{k}|^{p_{k}})\bigr)^{1/p_{k}}.
\end{align*}
\end{proof}

\section{A Minkowski Inequality}\label{S:MI}

We employ the H\"older inequality in Theorem~\ref{T:HI}(1) to prove a Minkowski inequality in the setting of Archimedean $\Phi$-algebras over $\mathbb{K}$ in this section.

\begin{theorem}\label{T:MI} \textnormal{\textbf{(Minkowski Inequality)}}
Let $p\in(1,\infty)$. Suppose that $A$ and $B$ are both weighted geometric mean closed Archimedean $\Phi$-algebras over $\mathbb{K}$. For every positive linear map $T\colon A\rightarrow B$, we have
\begin{align*}
\Bigl(T\Bigl(\bigl|\sum\limits_{k=1}^{n}a_{k}\bigr|^{p}\Bigr)\Bigr)^{1/p}\leq\sum\limits_{k=1}^{n}\bigl(T(|a_{k}|^{p})\bigr)^{1/p}\ (a_{1},\dots,a_{n}\in A).
\end{align*}
\end{theorem}

\begin{proof}
We prove the result for $n=2$ and note that the rest of the proof follows from a standard induction argument. To this end, let $T\colon A\rightarrow B$ be a positive linear map, and assume that $a,b\in A$. Let $q\in(1,\infty)$ satisfy $q^{-1}+p^{-1}=1$. By Lemma~\ref{L:powerrules}(2) (first equality), Theorem~\ref{T:HI}(1) (second inequality), and Lemma~\ref{L:powerrules}(1) (third equality),
\begin{align*}
T(|a+b|^{p})&=T(|a+b|^{p-1}|a+b|)\\
&\leq T\bigl(|a+b|^{p-1}(|a|+|b|)\bigr)=T(|a+b|^{p-1}|a|)+T(|a+b|^{p-1}|b|)\\
&\leq T\bigl((|a+b|^{p-1})^{q}\bigr)^{1/q}T(|a|^{p})^{1/p}+T\bigl((|a+b|^{p-1})^{q}\bigr)^{1/q}T(|b|^{p})^{1/p}\\
&=T(|a+b|^{p})^{1/q}T(|a|^{p})^{1/p}+T(|a+b|^{p})^{1/q}T(|b|^{p})^{1/p}\\
&=T(|a+b|^{p})^{1/q}\bigl(T(|a|^{p})^{1/p}+T(|b|^{p})^{1/p}\bigr).
\end{align*}
Setting $f=T(|a+b|^{p})$ and $g=T(|a|^{p})^{1/p}+T(|b|^{p})^{1/p}$, we have $f\leq f^{1/q}g$. Next let $C$ be the vector $\Phi$-subalgebra of $B$ generated by $f,f^{1/p},f^{1/q}$, and $g$. Let $\omega\colon C\rightarrow\mathbb{R}$ be a nonzero multiplicative vector lattice homomorphism, so that $\omega(e)=1$, where $e$ is the unit element of $A$. Using Corollary~\ref{C:T(a^r)=T(a)^r}, we have
\[
\omega(f)\leq\omega(f^{1/q}g)=\omega(f^{1/q})\omega(g)=\omega(f)^{1/q}\omega(g).
\]
Thus if $\omega(f)\neq 0$ then $\omega(f)^{1/p}\leq\omega(g)$. Of course, $\omega(f)^{1/p}\leq\omega(g)$ also holds in the case that $\omega(f)=0$, since $g\in B^{+}$. By Corollary~\ref{C:T(a^r)=T(a)^r} again, $\omega(f^{1/p})\leq\omega(g)$. Since the collection of all nonzero multiplicative vector lattice homomorphisms separate the points of $C$ (see \cite[Corollary 2.7]{BusdPvR}), we conclude that $f^{1/p}\leq g$. Therefore, we obtain
\[
T(|a+b|^{p})^{1/p}\leq T(|a|^{p})^{1/p}+T(|b|^{p})^{1/p}.
\]
\end{proof}

\bibliography{mybib}

\providecommand{\bysame}{\leavevmode\hbox to3em{\hrulefill}\thinspace}
\providecommand{\MR}{\relax\ifhmode\unskip\space\fi MR }
\providecommand{\MRhref}[2]{%
  \href{http://www.ams.org/mathscinet-getitem?mr=#1}{#2}
}
\providecommand{\href}[2]{#2}
\begin{thebibliography}{10}

\bibitem{AB}
C.D. Aliprantis and O.~Burkinshaw, \emph{Positive {O}perators}, Academic Press,
  Orlando, 1985.

\bibitem{Az}
Y.~Azouzi, \emph{Square {M}ean {C}losed {R}eal {R}iesz {S}paces}, Ph.D. thesis,
  Tunis, 2008.

\bibitem{AzBoBus}
Y.~Azzouzi, K.~Boulabiar, and G.~Buskes, \emph{The de {S}chipper formula and
  squares of {R}iesz spaces}, Indag. Math. (N.S.) \textbf{17} (2006), no.~4,
  479--496.

\bibitem{BerHui2}
S.~J. Bernau and C.~B. Huijsmans, \emph{Almost {$f$}-algebras and
  {$d$}-algebras}, Math. Proc. Cambridge Philos. Soc. \textbf{107} (1990),
  no.~2, 287--308.

\bibitem{BerHui}
\bysame, \emph{The {S}chwarz inequality in {A}rchimedean {$f$}-algebras},
  Indag. Math. (N.S.) \textbf{7} (1996), no.~2.

\bibitem{BeuHuidP}
F.~Beukers, C.~Huijsmans, and B.~de~Pagter, \emph{Unital embedding and
  complexification of f-algebras}, Math Z. \textbf{183} (1983), no.~1,
  131--144.

\bibitem{BeuHui}
F.~Beukers and C.~B. Huijsmans, \emph{Calculus in {$f$}-algebras}, J. Austral.
  Math. Soc. Ser. A \textbf{37} (1984), no.~1, 110--116.

\bibitem{Bo2}
K.~Boulabiar, \emph{A {H}\"older-type inequality for positive functionals on
  {$\Phi$}-algebras}, J. Inequal. Pure Appl. Math. \textbf{3} (2002), no.~5,
  Article 74, 7 pp. (electronic).

\bibitem{BoBus}
K.~Boulabiar and G.~Buskes, \emph{Vector lattice powers: f-algebras and
  functional calculus}, Comm. Algebra \textbf{34} (2006), no.~4, 1435--1442.

\bibitem{BusdPvR}
G.~Buskes, B.~de~Pagter, and A.~van Rooij, \emph{Functional calculus on {R}iesz
  spaces}, Indag. Math. (N.S.) \textbf{2} (1991), no.~4, 423--436.

\bibitem{BusSch2}
G.~Buskes and C.~Schwanke, \emph{Complex vector lattices via functional
  completions}, J. Math. Anal. Appl. \textbf{434} (2016), no.~2, 1762--1778.

\bibitem{BusSch}
\bysame, \emph{Functional completions of {A}rchimedean vector lattices},
  Algebra Universalis \textbf{76} (2016), no.~1, 53--69.

\bibitem{BusvR4}
G.~Buskes and A.~van Rooij, \emph{Almost f-algebras: commutativity and the
  {C}auchy-{S}chwarz inequality}, Positivity \textbf{4} (2000), no.~3,
  227–--231.

\bibitem{BusvR2}
\bysame, \emph{Squares of {R}iesz spaces}, Rocky Mountain J. Math. \textbf{31}
  (2001), no.~1, 45--56.

\bibitem{Con}
J.B. Conway, \emph{A {C}ourse in {F}unctional {A}nalysis}, Springer-Verlag, New
  York, 1990.

\bibitem{Kus2}
A.~G. Kusraev, \emph{H\"older type inequalities for orthosymmetric bilinear
  operators}, Vladikavkaz. Mat. Zh. \textbf{9} (2007), no.~3, 36--46.

\bibitem{LuxZan1}
W.~A.~J. Luxemburg and A.~C. Zaanen, \emph{{R}iesz {S}paces {V}ol. {I}},
  North-Holland Publishing Co., Amsterdam-London; American Elsevier Publishing
  Co., New York, 1971.

\bibitem{Mali}
L.~Maligranda, \emph{A simple proof of the {H}\"older and the {M}inkowski
  inequality}, Amer. Math. Monthly \textbf{102} (1995), no.~3, 256--259.

\bibitem{MW}
G.~Mittelmeyer and M.~Wolff, \emph{\"{U}ber den {A}bsolutbetrag auf komplexen
  {V}ektorverb\"anden}, Math. Z. \textbf{137} (1974), 87--92.

\bibitem{dP}
B.~de Pagter, \emph{{f}-algebras and {O}rthomorphisms}, Ph.D. thesis, Leiden,
  1981.

\bibitem{dS}
W.~J.~de Schipper, \emph{A note on the modulus of an order bounded linear
  operator between complex vector lattices}, Nederl. Akad. Wetensch. Proc. Ser.
  A \textbf{76} (1973), 355--367.

\bibitem{Tou}
M.~A. Toumi, \emph{Exponents in {$\mathbb{R}$} of elements in a uniformly
  complete {$\Phi$}-algebra}, Rend. Istit. Mat. Univ. Trieste \textbf{40}
  (2008), 29--44 (2009).

\bibitem{Zan2}
A.C. Zaanen, \emph{Riesz {S}paces {II}}, North-Holland Mathematical Library,
  vol.~30, North-Holland Publishing Co., Amsterdam, 1983.

\end{thebibliography}
\bibliographystyle{amsplain}

\end{document}